\documentclass{amsart}
\usepackage{amssymb}
\usepackage{amsmath}
\usepackage{version}

\newtheorem{theorem}{Theorem}[section] 
\newtheorem{lemma}[theorem]{Lemma}     
\newtheorem{corollary}[theorem]{Corollary}

\newtheorem{remark}{Remark}

\newcommand{\T}{\mathbb{T}}
\newcommand{\D}{\mathbb{D}}
\newcommand{\E}{\mathbb{E}}
\newcommand{\la}{\lambda}
\newcommand{\al}{\alpha}
\newcommand{\ph}{\varphi}

\newcommand{\up}{\upsilon}
\newcommand{\C}{\mathbb{C}}
\newcommand{\R}{\mathbb{R}}
\newcommand{\G}{\mathbb{G}}
\newcommand{\nn}{\nonumber}
\newcommand{\ds}{\displaystyle}
\newcommand{\dia}{\diamond}
\renewcommand{\S}{\mathcal{S}}
\providecommand{\Aut}{\mathop{\rm Aut}\nolimits}

\numberwithin{equation}{section}
\numberwithin{theorem}{section}
\numberwithin{remark}{section}

\title{The automorphism group of the tetrablock} 

\author{ N. J. Young}
 
\begin{document}
\begin{abstract}
The tetrablock is shown to be inhomogeneous and its automorphism group is determined.  A type of Schwarz lemma for the tetrablock is proved.  The action of the automorphism group is described in terms of a certain natural foliation by complex geodesic discs.
\end{abstract}
\maketitle
 \footnotetext{\textit{Math. Subject Classifications}: 32M12 (primary), 30C80, 93D21 (secondary)}

\section{Introduction} \label{intro}
The {\em tetrablock} is the domain $\E$ in $\C^3$ defined by
\[
\E= \{ (x_1,x_2,x_3)\in\C^3: 1-x_1z-x_2 w +x_3 zw \neq 0 \mbox{ for all } z,w\in\C \mbox{ such that } |z|\leq1, \, |w| \leq 1 \}.
\]
$\E$ is a non-convex domain whose intersection with $\R^3$ is a regular tetrahedron.  It is of interest because of its relation to a certain function-theoretic problem that arises in control engineering; see Section \ref{conclud} below.  In this paper we answer three questions: is $\E$ homogeneous? Is $\E$ an analytic retract of the unit ball of the space of $2\times 2$ matrices? What is the full group of automorphisms of $\E$?  Here
an {\em automorphism} of a domain $\Omega$ is an analytic bijective self-map of $\Omega$ having an analytic inverse. 

  In Section \ref{schwarz} we prove a Schwarz lemma for $\E$:  we find necessary and sufficient conditions on $y\in\C^3$ for the existence of an analytic map $\ph:\D\to\E$, where $\D$ is the open unit disc, such that $\ph(0)=(0,0,0)$ and $\ph'(0)=y$ and we give a formula for a suitable $\ph$.  This result enables us to show in Section \ref{neous} that $\E$ is inhomogeneous; the proof uses E. Cartan's classification of bounded homogeneous domains in $\C^3$ and a little elementary theory of $J^*$-algebras.  We also show that $\E$ is not an analytic retract of any bounded symmetric homogeneous domain of dimension less than 16.  In Section \ref{aut} we determine the automorphism group of $\E$, thereby verifying a conjecture made in \cite{awy}.   In Section \ref{leaves} we show that the action of the automorphism group of $\E$ can be understood in terms of a certain natural foliation of $\E$ by analytic discs; the group permutes the leaves of this foliation transitively, and the orbits of the group are naturally parametrised by the interval $[0,1)$.

 The connection between the geometry of $\E$ and the problem of ``$\mu$-synthesis'' from control engineering is outlined in \cite[Section 9]{awy} and references cited there.

We shall denote the closure of $\E$ by $\bar\E$ and the closed unit disc by $\Delta$.  We write $O$ for the origin $(0,0,0)$ in $\C^3$.  The automorphism group of a domain $\Omega$ will be denoted by $\Aut\Omega$.  If $H,K$ are Hilbert spaces then $\mathcal{L}(H,K)$ denotes the linear space of bounded linear operators from $H$ to $K$ with the operator norm.   $\C^{2\times 2}$ denotes the space of $2\times 2$ complex matrices with the standard $C^*$ norm.  An important role in the analysis of $\E$ is played by the map
\begin{equation} \label{defpi}
\pi: \C^{2\times 2} \to \C^3:[a_{ij}] \mapsto (a_{11},a_{22}, \det[a_{ij}]).
\end{equation}
We shall write $\S_{2\times 2}$ for the set of analytic functions $F:\D\to \C^{2\times 2}$ such that $||F(\la)|| < 1$ for all $\la\in\D$.

Let us recapitulate here some of the eleven different characterizations of $\E$ from \cite{awy}.  For present purposes four will suffice.  One of them uses the rational function
\begin{equation} \label{defPsi}
\Psi(z,x_1,x_2,x_3) = \frac{x_3z-x_1}{x_2z-1}.
\end{equation}
Roughly speaking we identify $x\in\C^3$ with the linear fractional transformation $\Psi(.,x)$; then $x\in\E$ if and only if $x$ corresponds to a linear fractional transformation whose supremum on $\D$ is less than one.  This statement is not quite precise, since if $x_1x_2=x_3$ then $\Psi(.,x)$ is constant and equal to $x_1$.  Points $x$ for which $x_1x_2=x_3$ are called {\em triangular points}; they require special treatment.
\begin{theorem}\label{recap}
For $x=(x_1,x_2,x_3)\in\C^3$ the following are equivalent.
\begin{enumerate}
\item $x\in\E$;
\item $|x_1-\bar x_2x_3|+ |x_1x_2-x_3| < 1-|x_2|^2$;
\item $\sup_{z\in\D}|\Psi(z,x)| < 1$ and if $x_1x_2=x_3$ then, in addition, $|x_2|<1$;
\item there exists a symmetric matrix $A\in\C^{2\times 2}$ such that $||A||<1$ and $\pi(A)=x$;
\item there exist $\beta_1, \beta_2\in\C$ such that $|\beta_1|+|\beta_2| < 1$ and
\[
x_1=\beta_1 + \bar\beta_2 x_3, \quad x_2=\beta_2+\bar\beta_1 x_3.
\]
\end{enumerate}
\end{theorem}
For the proof see \cite[Theorem 2.1]{awy}.  
Some basic complex geometry of $\E$ is described in this reference.  For example, $\E$ is starlike about the origin, is polynomially convex, has a distinguished boundary of $3$ real dimensions and admits a group of automorphisms of $6$ real dimensions. 

Condition (iv) reveals a close connection between $\E$ and the two Cartan domains $R_I(2,2)$ and $R_{II}(2)$, defined to be the open unit balls in the spaces of $2\times 2$ matrices and symmetric $2\times 2$ matrices respectively:
\[
\E= \pi(R_I(2,2)) = \pi(R_{II}(2)).
\]
The homogeneity of $R_I(2,2)$ was used in \cite{awy} to prove a Schwarz lemma for $\E$, that is, a criterion for the solvability of certain 2-point interpolation problems for analytic functions from $\D$ to $\E$.  In the next section a similar method is used to prove the other sort of Schwarz lemma for $\E$: a criterion is found for the existence of an analytic function from $\D$ to $\E$ with a prescribed value and derivative at a single point.

\section{A Schwarz lemma for the tetrablock} \label{schwarz}
\begin{theorem} \label{newschw}
Let $y\in\C^3$.  There exists an analytic map $\ph: \D\to\E$ such that $\ph(0)= O$ and $\ph'(0)=y$ if and only if
\begin{equation} \label{ycond}
\max\{|y_1|,|y_2|\} + |y_3| \leq 1.
\end{equation}
\end{theorem}
\begin{proof}
Suppose such a $\ph$ exists.  Write $\ph=(\ph_1,\ph_2,\ph_3)$.  By 
\cite[Theorem 1.2]{awy}, for any $\la\in\D$,
\[
\max\left\{\frac{|(\ph_1-\bar\ph_2\ph_3)(\la)|+|(\ph_1\ph_2-\ph_3)(\la)|}{1-|\ph_2(\la)|^2},\frac{|(\ph_2-\bar\ph_1\ph_3)(\la)|+|(\ph_1\ph_2-\ph_3)(\la)|}{1-|\ph_1(\la)|^2} \right\} \leq |\la|.
\]
Divide through by $|\la|$ and let $\la\to 0$ to obtain
\[
\max\{|\ph_1'(0)|+|\ph_3'(0)|, |\ph_2'(0)|+|\ph_3'(0)|\} \leq 1.
\]
Since $\ph'(0)=y$ we have
\[
\max\{|y_1|,|y_2|\}+|y_3| \leq 1,
\]
and the inequality (\ref{ycond}) is necessary for the existence of $\ph$.

Conversely, suppose that (\ref{ycond}) holds.  We can suppose that $|y_1| \geq |y_2|$.
If $y_1=0$ then also $y_2=0$, $|y_3| \leq 1$  and the function $\ph(\la)=(0,0,\la y_3)$ is sufficient.  We may therefore assume that $y_1 \neq 0$.  We shall construct $F\in\S_{2\times 2}$ such that $\ph = \pi\circ F$ has the desired properties.  Note that since $\pi(R_I(2,2))=\E$, for $F\in\S_{2\times 2}$,  $\ph$ maps $\D$ into $\E$.
Let $\zeta\in\D$ be a number to be chosen later and let $F=[F_{ij}]$ satisfy
\begin{equation} \label{valF0}
F(0)= \left[\begin{array}{cc} 0& \zeta \\ 0& 0 \end{array}\right].
\end{equation}
Then $\ph(0)=O$ and
\begin{eqnarray*}
\ph'(0)&=& (\pi\circ F)'(0)=(F_{11}', F_{22}', F_{11}'F_{22}+F_{11}F_{22}'-F_{12}'F_{21}-F_{12}F_{21}')(0)\\
&=&(F_{11}', F_{22}', -\zeta F_{21}')(0).
\end{eqnarray*}
Thus $\ph'(0)= y$ if and only if
\begin{equation} \label{valFdash0}
F'(0) = \left[\begin{array}{cc} y_1 & * \\ -y_3/\zeta & y_2 \end{array}\right].
\end{equation}
Accordingly our task is to find $\zeta\in\D$ and $F\in\S_{2\times 2}$ such that the equations (\ref{valF0}) and (\ref{valFdash0}) are satisfied.

We shall use the matricial M\"obius transformation $\mathcal{M}_Z$ of $2\times 2$ matrices defined, for any strict $2\times 2$ contraction $Z$ by
\[
\mathcal{M}_Z(X) = -Z+ D_{Z^*}X(1-Z^*X)^{-1}D_Z,
\]
where $D_Z=(1-Z^*Z)^{\tfrac 12}$. The transformation $\mathcal{M}_Z$ is an automorphism of the unit ball $R_I(2,2)$ of $\C^{2\times 2}$, has inverse
$\mathcal{M}_{-Z}$ and maps $Z$ to $0$.  We have, for any $F\in\S_{2\times 2}$,
\begin{eqnarray}
(\mathcal{M}_Z\circ F)' &= & D_{Z^*}[F'(1-Z^*F)^{-1}+F(1-Z^*F)^{-1}Z^*F'(1-Z^*F)^{-1}]D_Z  \nn \\
	&=& D_{Z^*}(1-FZ^*)^{-1}F'(1-Z^*F)^{-1}D_Z.
\end{eqnarray}
Let 
\begin{equation} \label{defZ}
Z=\left[\begin{array}{cc} 0 & \zeta \\ 0 & 0 \end{array} \right]
\end{equation}
for some $\zeta\in\D$.
  Then $Z$ is a strict contraction and
\[
D_Z=\left[\begin{array}{cc} 1&0\\0&(1-|\zeta|^2)^{\tfrac 12} \end{array}\right], \quad 
D_{Z^*}=\left[\begin{array}{cc} (1-|\zeta|^2)^{\tfrac 12}&0\\0& 1\end{array}\right].
\]
Hence, if $F$ satisfies equations (\ref{valF0}) and (\ref{valFdash0}), then
\begin{eqnarray} \label{strcon}
(\mathcal{M}_Z\circ F)' (0)&= & D_{Z^*}^{-1}F'(0)D_Z^{-1}\nn \\
	&=& \left[\begin{array}{cc} \ds \frac{y_1}{(1-|\zeta|^2)^{\tfrac 12}} & \ds\frac{F'_{12}(0)}{1-|\zeta|^2} \\
 {}&     \\
{\ds -\frac{y_3}{\zeta} } &\ds \frac{y_2}{(1-|\zeta|^2)^{\tfrac 12}}  
\end{array} \right] .
\end{eqnarray}
If the required $F$ exists then, by the Schwarz Lemma for $R_I(2,2)$, the right hand side of equation (\ref{strcon}) is a strict contraction.  We shall show that $F$ exists by working back from equation (\ref{strcon}).

The choice $\zeta= \sqrt{1-|y_1|}$ in equation (\ref{strcon}) leads us to define
\begin{equation}\label{defY}
Y(\xi) = \left[ \begin{array}{cc} \ds \frac{y_1}{|y_1|^{\tfrac 12}} & \xi \\
	& \\
	\ds -\frac{y_3}{\sqrt{1-|y_1|}} &\ds \frac{y_2}{|y_1|^{\tfrac 12}} \end{array} \right]
\end{equation}
for some $\xi\in\C$.
  Since, by hypothesis, $|y_3| \leq 1 - |y_1|$, the first column of $Y(\xi)$ has norm
\[
 \left\{ |y_1| + \frac{|y_3|^2}{1-|y_1|} \right\}^{\tfrac 12} \leq
\left\{ |y_1| + \frac{|y_3|(1-|y_1|)}{1-|y_1|} \right\}^{\tfrac 12} = \{|y_1| + |y_3|\}^{\tfrac 12} \leq 1,
\]
and since $|y_2| \leq |y_1|$, the second row of $Y(\xi)$ also has norm at most $1$.  By Parrott's Theorem (\cite{par} or \cite[Theorem 12.22]{Y}) there exists $\xi \in\C$ such that $||Y(\xi)|| < 1$; in fact, a suitable choice is 
\begin{equation} \label{defxi}
\xi= \frac{y_1y_2\bar y_3\sqrt{1-|y_1|}}{|y_1|(1-|y_1|-|y_3|^2)}.
\end{equation}
Let $H(\la)=\la Y(\xi), \la\in\D$.  Then $H \in \S_{2\times 2}$ and
\[
H(0)=0, \quad H'(0) = Y(\xi).
\]
Define $F= \mathcal{M}_{-Z}\circ H$, where as before $Z$ is the right hand side of equation (\ref{valF0}), and now $\zeta=\sqrt{1-|y_1|}$.  Then $F\in\S_{2\times 2}$,
\[
F(0) = \mathcal{M}_{-Z}(0) = Z = \left[\begin{array}{cc} 0& \zeta \\ 0&0 \end{array}
\right]
\]
and
\begin{eqnarray*}
F'(0)&=&(\mathcal{M}_{-Z}\circ H)'(0) = (D_{Z^*}(1+HZ^*)^{-1}H'(1+Z^*H)^{-1}D_Z)(0)\\
	&=& D_{Z^*}Y(\xi) D_{Z}\\
	&=& \left[\begin{array}{cc} |y_1|^{\tfrac 12}&0 \\ 0 & 1\end{array}\right]
	\left[ \begin{array}{cc} \ds \frac{y_1}{|y_1|^{\tfrac 12}} & \xi \\
	\ds -\frac{y_3}{\zeta} &\ds \frac{y_2}{|y_1|^{\tfrac 12}} \end{array} \right] \left[\begin{array}{cc} 1 &0 \\ 0 & |y_1|^{\tfrac 12}\end{array}\right] \\
	&=& \left[ \begin{array}{cc} y_1 & \xi|y_1| \\ -y_3/\zeta & y_2 \end{array}\right].
\end{eqnarray*}
On comparison with equations (\ref{valF0}) and (\ref{valFdash0}) we find that $\ph=\pi\circ F$ satisfies the requirements of the theorem.
\end{proof}
We can extract from the above proof an explicit formula for $\ph$ satisfying the conditions of the theorem.
\begin{theorem} \label{formphi}
 Let $y\in\C^3$ be such that  $\max\{|y_1|,|y_2|\}+|y_3|\leq 1$.
Let $\ph:\D \to \C^3$ be given by 
\begin{equation}     \label{defph}
\ph(\la)=  \frac{\la }{1+\la\bar y_3C(y)}(y_1,y_2,C(y)\la+y_3)
\end{equation}
where
\begin{equation}\label {defC}
 C(y)= \left\{\begin{array}{clc} 0 & \mbox{ if } & y_1=y_2=0 \\
\ds\frac{y_1y_2(1-|y_1|)}{|y_1|(1-|y_1|-|y_3|^2)} & \mbox{ if } &  |y_2| \leq |y_1| \neq 0 \\
	& & \\
	\ds\frac{y_1y_2(1-|y_2|)}{|y_2|(1-|y_2|-|y_3|^2)} & \mbox{ if } &  |y_1| \leq |y_2| \neq 0. \end{array}
\right.
\end{equation}
Then $\ph$ is an analytic map from $\D$ to $\E$, $\ph(0)=O$ and $\ph'(0)=y$.
\end{theorem}
\begin{proof}
We considered the case $y_1=y_2=0$ in the proof of Theorem \ref{newschw}.  Suppose without loss that $|y_2|\leq |y_1|\neq 0$.  It is immediate that $\ph$ as defined satisfies $\ph(0)=0,\ph'(0)=y$; the task is to show that $\ph$ is analytic and $\ph(\D) \subset\E$.

Choose $\zeta= \sqrt{1-|y_1|}$ and  $Z, \xi, Y=Y(\xi)$ as in equations (\ref{defZ}), (\ref{defxi}),(\ref{defY}) respectively, and let
\begin{eqnarray*} 
F(\la) &=& \mathcal{M}_{-Z}(\la Y) = Z+D_{Z^*}\la Y(1+\la Z^*Y)^{-1}D_Z \nn \\
	&=& Z+ \la (D_{Z^*}YD_Z) (1+\la Z^*D_{Z^*}^{-1}YD_Z)^{-1}.
\end{eqnarray*}
As we observed in the proof of Theorem \ref{newschw}, $F\in\S_{2\times 2}$.
We have
\[
D_{Z^*}YD_Z = \left[\begin{array}{cc} y_1 & \xi|y_1| \\ -y_3/\zeta & y_2 \end{array}\right], \quad
D_{Z^*}^{-1}YD_Z = \left[\begin{array}{cc} \frac{y_1}{|y_1|} & \xi \\ -y_3/\zeta & y_2 \end{array}\right].
\]
Hence
\begin{eqnarray} \label{formF}
F(\la) &=& Z + \la (D_{Z^*}YD_Z)
\left( 1 +\la \left[\begin{array}{cc} 0 & 0  \\ \zeta & 0\end{array}\right]
\left[\begin{array}{cc} \frac{y_1}{|y_1|} & \xi \\ -y_3/\zeta & y_2 \end{array}\right]\right)^{-1} \nn\\
	&=& Z+ \frac{\la}{1+\la\xi\zeta}
\left[\begin{array}{cc} y_1 & \xi|y_1| \\ -y_3/\zeta & y_2 \end{array}\right]\left[\begin{array}{cc} 1+\la\xi\zeta &0\\ -\frac{\la\zeta y_1}{|y_1|} & 1\end{array}\right] \nn\\
	&=&  \left[\begin{array}{cc} 0 & \zeta \\ 0 & 0\end{array}\right] + \frac{\la}{1+\la\xi\zeta} \left[\begin{array}{cc} y_1 & \xi|y_1| \\ w &  y_2  \end{array}\right]
\end{eqnarray}
where
\begin{eqnarray} \label{defw}
w&=&w(\la)=-\frac{y_3}{\zeta}(1+\la\xi\zeta) - \frac{\la\zeta y_1y_2}{|y_1|}
	= -\frac{y_3}{\zeta} -\frac{\la\zeta^3y_1y_2}{|y_1|(1-|y_1|-|y_3|^2)} \nn \\
	&=&-\frac{y_3}{\sqrt{1-|y_1|}} -\la\sqrt{1-|y_1|}C(y).
\end{eqnarray}
We find that
\[
\det F(\la) =\frac{\la^2(y_1y_2-w\xi |y_1|)}{(1+\la\xi\zeta)^2} -
 \frac{\la w\sqrt{1-|y_1|}}{1+\la\xi\zeta}.
\]
Note that 
\[
\xi\zeta = \frac{y_1y_2\bar y_3(1-|y_1|)}{|y_1|(1-|y_1|-|y_3|^2)}=\bar y_3C(y).
\]
We have 
\[
y_1y_2 -w(\la)\xi|y_1|= \frac{y_1y_2(1-|y_1|)}{1-|y_1|-|y_3|^2} +
	\frac{\la \bar y_3}{|y_1|} \left(\frac{y_1y_2(1-|y_1|)}{1-|y_1|-|y_3|^2}\right)^2=|y_1|C(y)(1+\la \bar y_3 C(y)),
\]
 and so
\begin{eqnarray*}
\det F(\la)&=& \frac{\la}{1+\la\bar y_3C(y)}\left(y_3+\la (1-|y_1|)C(y)\right) +
 \left(\frac{\la}{1+\la\bar y_3C(y)}\right)^2|y_1|C(y)(1+\la\bar y_3C(y))\\
	&=&\frac{\la}{1+\la\bar y_3C(y)} (C(y)\la+y_3).
\end{eqnarray*}
Since $F\in\S_{2\times 2}$ the map 
$\pi\circ F$ is analytic, maps $\D$ to $\E$ and satisfies (compare equation (\ref{formF}))
\begin{eqnarray*}
(\pi\circ F)_1(\la)&=& F_{11}(\la) =  \frac{\la y_1}{1+\la\bar y_3C(y)}, \\
(\pi\circ F)_2(\la)&=& F_{22}(\la) =  \frac{\la y_2}{1+\la\bar y_3C(y)}, \\
(\pi\circ F)_3(\la)&=& \det F(\la) = \frac{\la}{1+\la\bar y_3C(y)}(C(y)\la+y_3).
\end{eqnarray*}
Comparison with equation (\ref{defph}) shows that $\pi\circ F= \ph$, and hence $\ph$ has the required properties.
\end{proof}
\begin{remark}
\rm
 In the event that the necessary condition of Theorem \ref{newschw} holds with equality, that is,
\[
\max\{|y_1|,|y_2|\} + |y_3| =1,
\]
the function $\ph$ of Theorem \ref{formphi} is a complex geodesic of $\E$ (that is, it has an analytic left inverse).  Suppose that $|y_2| \leq |y_1| \neq 0$ and $|y_1| + |y_3| =1$.
Choose $\omega_1, \omega_3 \in\T$ such that $\omega_1y_1=|y_1|,\omega_3y_3=|y_3|$; then $\omega_1y_1+\omega_3y_3=1$.  For any $z\in\Delta$ the rational function $\Psi(z,.)$ given by equation (\ref{defPsi}) maps $\E$ analytically into $\D$.  We have
\[
\Psi(\omega, \ph(\la))=\frac{\ph_3(\la)\omega - \ph_1(\la)}{\ph_2(\la)\omega -1}
	=\la\frac{(C\la+y_3)\omega - y_1}{\la y_2\omega -(1+C\bar y_3\la)},
\]
where (since $1-|y_1|=|y_3|$)
\[
C=C(y)= \frac{y_1y_2|y_3|}{|y_1|(|y_3|-|y_3|^2)}=\frac{y_1y_2}{|y_1|^2}
	= \frac{y_2}{\bar y_1}.
\]
Choose $\omega= -\bar\omega_1\omega_3$.  A little calculation gives the relation
\[
\Psi(\omega, \ph(\la)) = \bar\omega_1\la.
\]
Hence $\omega_1\Psi(\omega,.):\E\to\D$ is an analytic left inverse of $\ph$, and so $\ph$ is a complex geodesic of $\E$.

One might expect (by analogy with the case of the unit disc) that in the extremal case $\ph$ should be $\E$-inner \cite{alaa}, that is, the radial limit function of $\ph$  should map $\T$ almost everywhere into the distinguished boundary $b\E$ of $\E$.    In fact $b\E$ is the intersection of the closure $\bar\E$ of $\E$ with the set $\{x\in\C^3:|x_3|=1\}$ \cite[Theorem 7.1]{awy}, and so an analytic map $\ph:\D\to\E$ is $\E$-inner if and only if $\ph_3$ is a scalar inner function.
For the function $\ph$ of the theorem, $\ph_3$ is inner if and only if  $y$ is ``doubly extremal'', that is, $|y_1|=|y_2|=1-|y_3|$.
\end{remark}

\section{The tetrablock is inhomogeneous} \label{neous}
 We shall show that the inhomogeneity of $\E$  follows from Theorem \ref{newschw} and E. Cartan's classification of bounded homogeneous domains \cite[page 313]{fuks}.  We use L. A. Harris' theory of $J^*$-algebras\cite{harris}.   A $J^*$-algebra is a closed subspace $\mathcal{A}$ of the Banach space $\mathcal{L}(H,K)$, for some Hilbert spaces $H,K$, with the property that $T\in \mathcal{A}$ implies $TT^*T \in \mathcal{A}$.  The importance of such algebras here is that, in dimensions up to 15, every bounded symmetric homogeneous domain is isomorphic to the open unit ball of a $J^*$-algebra.  A domain $\Omega$ is said to be {\em symmetric} if, for every $z\in\Omega$, there is an analytic involution of $\Omega$ of which $z$ is an isolated fixed point.

A domain $\Omega_1$ is said to be an {\em analytic retract} of a domain $\Omega_2$ if there exist analytic maps $h:\Omega_1\to\Omega_2, \kappa:\Omega_2\to\Omega_1$ such that $\kappa\circ h=\mathrm{id}_\E$.  We define the {\em indicatrix} $I(\Omega,a)$ of a domain $\Omega$ at a point $a\in\Omega$ to be the set
\[
I(\Omega,a)=\{\ph'(0): \ph \mbox{ is an analytic map from }\D \mbox{ to } \Omega, \ph(0)=a\}.
\]
It follows from the chain rule that if $h:\Omega_1\subset\C^n\to\Omega_2$ is analytic and $a\in\Omega_1$ then $h'(a)I(\Omega_1,a) \subset I(\Omega_2,h(a))$.  If, further, $h$ has an analytic left inverse $\kappa$, then $\kappa'\circ h(a)h'(a)$ is the identity operator
on $\C^n$, and so $\kappa'\circ h(a)$ is a linear operator that maps $I(\Omega_2, h(a))$ surjectively onto $I(\Omega_1,a)$.

We recall \cite{harris2} that the {\em rank} of  a $J^*$-algebra $\mathcal{A}$ is the supremum of the number of non-zero elements in the spectrum of $T^*T$ over all $T\in \mathcal{A}$; it is also equal to the maximum cardinality of any set of mutually orthogonal non-zero minimal partial isometries in $\mathcal{A}$ \cite [Corollary 5]{harris2}.  Every finite-dimensional $J^*$-algebra clearly has finite rank.
\begin{theorem}  \label{noretract}
$\E$ is not an analytic retract of the open unit ball of any  $J^*$-algebra of finite rank.
\end{theorem}
\begin{proof}
Let $\mathcal{A}\subset\mathcal{L}(H,K)$ be a $J^*$-algebra of rank $r < \infty$, $B$ its open unit ball.  Suppose that $h:\E\to B, \kappa: B\to\E$ are analytic and $\kappa\circ h = \mathrm{id}_\E$.  Since the open unit ball of a $J^*$-algebra is homogeneous, we may replace $h, \kappa$ by their compositions with automorphisms of $B$ to ensure that $h(0)=0$.
By Theorem \ref{newschw},
\[
 I(\E,0)=\{y\in\C^3: \max\{|y_1|,|y_2|\}+|y_3| \leq 1\}.
\]
 It is easy to see that $I(B,0)$ is the closed unit ball $\bar B$ of $\mathcal{A}$.  Indeed, if $T\in\bar B$ then the function $\ph(\la)=\la T$ maps $\D$ to $B$ and $0$ to $0$ and satisfies $\ph'(0)=T$, so that $I(B,0) \supset \bar B$, while if $\ph:\D\to B$ is analytic and maps $0$ to $0$ then the Schwarz lemma for $\D$ applied to the scalar functions $\left<\ph(.)\xi,\eta\right>, \xi\in H, \eta\in K$ shows that $\ph'(0) \in \bar B$.

$I(\E,0)$ is the closed unit ball of $\C^3$ with respect to the norm 
\[
||y||_\E = \max\{|y_1|,|y_2|\}+|y_3|.
\]
Since the linear operators $h'(0):\C^3\to\mathcal{A}, \kappa'(0):\mathcal{A} \to \C^3$ are contractions with respect to $||.||_\E, ||.||_\mathcal{A}$ and  $\kappa'(0)h'(0)$ is the identity operator on $\C^3$, it follows that $h'(0)$ is an isometry.  Let $h'(0)(1,0,0) = A, h'(0)(0,0,1)=B$.  Then for any $\la,\mu\in\C$,
\begin{equation} \label{isom}
 ||\la A+\mu B||_{\mathcal{L}(H,K)} = ||h'(0)(\la,0,\mu)|| =||(\la,0,\mu)||_\E = |\la| + |\mu|;
\end{equation}
in particular, $||A||=||B||=1$.
By \cite[Proposition 4]{harris2}, every element $T\in \mathcal{A}$ has a singular value decomposition 
\[
T=\sum_{k=1}^m s_k V_k
\]
 where  $m \leq r$, each $s_k> 0$ 
  and the $V_k\in \mathcal{A}$ are mutually orthogonal non-zero minimal partial isometries.   Let $\omega_1,\omega_2,\dots$ be an infinite sequence of distinct points in $\T$.  Since $||A+\omega_1B||= 2$ we can write down the singular value decomposition
\[
A+\omega_1 B = 2V_1+\dots+2V_{n_1} + R_1
\]
where $1\leq n_1\leq r, R_1 \in \mathcal{A}, ||R_1||<1$ and $R_1$ is orthogonal to $V_1,\dots, V_{n_1}$.  The space of maximising vectors of $A+\omega_1B$ is
\[
M_1= \mathrm{span}\{V_1^*K,\dots,V_{n_1}^*K\}.
\]
For $x\in M_1$ we have 
\[
 2||x|| = ||Ax+\omega_1 Bx|| \leq ||Ax|| + ||Bx|| \leq 2||x||.
\]
It follows that $||Ax||=||x||=||Bx||$.  Moreover, the parallelogram law shows that $Ax=\omega_1Bx$.  Hence, for $x\in M_1$,
\[
 2Ax = (A+\omega_1B)x= 2V_1x+\dots+2V_{n_1}x.
\]
Thus we can write
\begin{eqnarray}
 A&=&V_1+\dots+V_{n_1} + A_1 \\
 \omega_1B&=&V_1+\dots+V_{n_1} + \omega_1B_1
\end{eqnarray}
where $A_1,B_1 \in \mathcal{A}$, $A_1$ and $B_1$ are both orthogonal to $V_1,\dots,V_{n_1}$ and $||A_1+\omega_1 B_1|| < 2$.  For any $\omega\in\T, \omega \neq \omega_1,$  $||A+\omega B||=2$  and
\[
A+\omega B = (1+\omega\bar\omega_1)(V_1+\dots+V_{n_1}) + A_1+\omega B_1.
\]
Since $\omega\neq\omega_1$ we have $|1+\omega\bar\omega_1| < 2$.  Hence $||A_1+\omega B_1|| = 2$ for any $\omega\in \T\setminus \{\omega_1\}$ and by the same arguments we have
\begin{eqnarray}
 A_1 &=& V_{n_1+1}+\dots+V_{n_2} +A_2,\nn\\
\omega_2B_1&=& V_{n_1+1}+\dots+V_{n_2} +\omega_2B_2
\end{eqnarray}
where $A_2,B_2 \in \mathcal{A}$, $A_2$ and $B_2$ are both orthogonal to $V_1,\dots,V_{n_2}$ and $||A_2+\omega_j B_2|| < 2, j=1,2$.
This process terminates after at most $r$ steps.  If we write $W_1= V_1+\dots+V_{n_1}$ etc. then, for some $N \leq r$,
\begin{eqnarray*}
 A &=& W_1+W_2+\dots+W_N,\\
 B &=& \bar\omega_1W_1 +\bar\omega_2W_2+\dots+\bar\omega_NW_N
\end{eqnarray*}
where $W_1,\dots,W_N$ are mutually orthogonal non-zero partial isometries in $\mathcal{A}$.  Choose $\omega\in\T$ different from $\omega_1,\dots,\omega_N$: then
\[
||A+\omega B|| = ||(1+\omega\bar\omega_1)W_1+\dots+(1+\omega\bar\omega_N)W_N|| = \max_{1\leq j \leq N} |1+\omega\bar\omega_j| < 2,
\]
 contrary to  equation (\ref{isom}).  Hence the postulated maps $h, \kappa$ do not exist.
\end{proof}

 \begin{corollary}
$\E$ is inhomogeneous.
\end{corollary}
\begin{proof}
E. Cartan showed that every bounded homogeneous domain in	 $\C^3$ is symmetric \cite[page 313]{fuks}.  Every bounded symmetric homogeneous domain in $\C^n, n< 15$, is the open unit ball of a $J^*$-algebra \cite[Theorem 7]{harris}.  $\E$ is bounded, and so if $\E$ is homogeneous then $\E$ is isomorphic to the open unit ball of a $3$-dimensional $J^*$-algebra, contrary to Theorem \ref{noretract}.  Hence $\E$ is inhomogeneous.
\end{proof}

\section{The automorphism group of the tetrablock} \label{aut}
Although the automorphism group $\Aut\E$ does not act transitively on $\E$, it is nevertheless quite large: there are commuting left and right
actions of $\Aut\D$ on $\E$ \cite[Theorem 6.8]{awy}.  These two actions together with the ``flip'' automorphism $F:(x_1,x_2,x_3)\mapsto(x_2,x_1,x_3)$ give a group $G$ of automorphisms of $\E$, and we conjectured in \cite{awy} that in fact $G=\Aut\E$.
In this section we prove that the conjecture is correct.

Roughly speaking, the actions of $\Aut\D$ on $\E$ are by composition.  Consider $x\in\bar\E$ and $y\in\E$.  The linear fractional maps $\Psi(.,x), \Psi(.,y)$ given by equation (\ref{defPsi}) map $\Delta$ into $\Delta, \D$ respectively \cite[Theorems 2.4 and 2.7]{awy}, and a simple calculation yields the relation
\[
\Psi(.,x)\circ \Psi(.,y) = \Psi(.,x\dia y)
\]
where
\begin{eqnarray} \label{defdia}
x\dia y &=& \frac{1}{1-x_2y_1} (x_1-x_3y_1,y_2-x_2y_3,x_1y_2-x_3y_3)  \\
	& = & \left(\Psi(y_1,x), \Psi(x_2,F(y)), \frac{x_1y_2-x_3y_3}{1-x_2y_1}\right). \nn
\end{eqnarray}
We define $x\dia y$ by equation (\ref{defdia}) for any $x,y\in\C^3$ such that $x_2y_1 \neq 1$.  Consider $\up \in\Aut\D$: we can write $\up = \Psi(.,\tau(\up))$ for some $\tau(\up)\in\bar\E$.  The left action of $\Aut\D$ on $\E$ or $\bar\E$ is given by $\up \cdot x= \tau(\up)\dia x$, or equivalently $\Psi(.,\up\cdot x) = \up \circ \Psi(.,x)$.  Similarly one defines a right action by $x\cdot\up = x\dia\tau(\up), x\in\E, \up\in\Aut\D$.  
Fuller details of the construction are given in \cite[Section 6]{awy}.  If $L_\up, R_\chi$ for any $\up, \chi\in\Aut\D$ are given by $L_\up(x) = \up\cdot x, R_\chi(x) = x\cdot \chi$ then $L_\up, R_\chi$ are commuting elements of $\Aut\E$.  Moreover, $L_\up L_\chi= L_{\up\circ\chi}$ and there is an involution $\up\mapsto \up_*$ on $\Aut\D$ such that $FL_\up= R_{\up_*}$ for any $\up\in\Aut\D$.  It follows that
\begin{equation} \label{defG}
G\stackrel{\rm def}{=}  \{L_\up R_\chi F^\nu: \up,\chi\in\Aut\D, \nu= 0 \mbox{ or }1 \}
\end{equation}
is a subgroup of $\Aut\E$.
\begin{theorem} \label{autE}
\[
 \Aut \E = G.
\]
\end{theorem}
The proof is based on the ideas of  M. Jarnicki and P. Pflug in their  determination of the automorphism group of the symmetrised bidisc \cite{JP};  the author and J. Agler had previously found a more elementary but longer proof of the same result.   
 An important role in the proof is played by the rotations $\rho_\omega\in\Aut\D$, defined by $\rho_\omega(z)=\omega z.$  It is easy to show that, for any $\omega\in\T$ and $x\in\E$,
\[
\rho_\omega\cdot x =(\omega x_1, x_2,\omega x_3), \quad x\cdot \rho_\omega = (x_1,\omega x_2,\omega x_3).
\]
\begin{lemma} \label{fixes}
Any automorphism of $\E$ that fixes every triangular point of $\E$ is the identity automorphism of $\E$.
\end{lemma}
\begin{proof}
Let $h\in\Aut\E$ fix all triangular points:  $h(x_1,x_2,x_1x_2)=(x_1,x_2,x_1x_2)$ for all $x_1,x_2 \in\D$.  We have
\begin{equation} \label{hdashO}
h'(O) = \left[\begin{array}{ccc} 1 & 0 & a \\ 0 & 1 & b \\ 0 & 0 & c\end{array}\right]
\end{equation}
for some $a,b,c \in \C$.
For $\omega\in \T$ let $H_\omega$ be the element $h^{-1}\circ L_{\rho_{1/\omega}} \circ h\circ L_{\rho_\omega}$ of $\Aut\E$.   Then $H_\omega(O)=(O)$ and
\[
H_\omega'(O)=h'(O)^{-1}\mathrm{diag}(\bar\omega,1,\bar\omega)h'(O)\mathrm{diag}(\omega,1,\omega)= \left[\begin{array}{ccc}1&0&0 \\ 0&1&b(\omega-1)\\
 0&0&1 \end{array}\right].
\]
If $H_\omega^n$ denotes the $n$th iterate of $H_\omega$ then
\[
(H_\omega^n)'(O) =H_\omega'(O)^n= \left[\begin{array} {ccc} 1&0&0\\0&1& nb(\omega-1)\\ 0&0&1 \end{array}\right].
\]
Now the isotropy group $K$ of the origin in $\E$, that is the group $\{f\in\Aut\E: f(O)=O\}$, is compact with respect to the topology of locally uniform convergence, the map $f\mapsto f'(O)$ is continuous on $\Aut\E$ and each $H_\omega^n \in K$.  Hence the matrices $(H_\omega^n)'(O), n\geq 1,$ are uniformly bounded.  It follows that $b=0$.  Similarly we have $a=0$ (replace $L_{\rho_\omega}$ by $R_{\rho_\omega}$ in the above argument).  Thus $h'(O)$ is diagonal and $H_\omega'(O)$ is the identity matrix.  By Cartan's theorem (e.g. \cite[Proposition 10.1.1]{krantz}) $H_\omega$ is the identity automorphism, and so $h\circ L_{\rho_\omega} = L_{\rho_\omega}\circ h$.  Similarly $h\circ R_{\rho_\omega} = R_{\rho_\omega}\circ h$.  If $h=(h_1,h_2,h_3)$ then, for $x\in\E$ and $\omega\in\T$,
\begin{eqnarray*}
h(\omega x_1,x_2,\omega x_3) &=& (\omega h_1(x), h_2(x),\omega h_3(x)), \\
h(x_1,\omega x_2,\omega x_3) &=& (h_1(x), \omega h_2(x), \omega h_3(x)).
\end{eqnarray*}
By the former equation, for fixed $x_2$, $h_2$ is homogeneous of degree $0$ in $x_1,x_3$ while $h_1,h_3$ are homogeneous of degree $1$ in $x_1,x_3$.  Similarly, for fixed $x_1$, $h_1$ is homogeneous of degree $0$ and $h_2,h_3$ are homogeneous of degree $1$ in $x_2,x_3$.  It follows that
\[
h(x) = (\al x_1, \beta x_2, \gamma x_1x_2 + \delta x_3)
\]
for some $\al,\beta,\gamma,\delta \in\C$.  Comparison with equation (\ref{hdashO}) shows that $\al=\beta=1, \delta=c \neq 0$.  Since $h$ fixes triangular points,  $\gamma=1-c$ and so $h(x)=(x_1,x_2,(1-c)x_1x_2+cx_3$).  We must prove that $c=1$.

Observe that  $h$ and $h^{-1}$ are polynomial maps and therefore extend continuously to $\bar\E$.  Hence $h$ induces an automorphism of the algebra $A(\E)$ of continuous scalar functions on $\bar\E$ that are analytic on $\E$, and consequently $h$ maps the Shilov boundary $b\E$ of $A(\E)$ to itself.  According to \cite[Theorem 7.1]{awy}, $x\in b\E$ if and only if $x_1=\bar x_2x_3, \ |x_2| \leq 1$ and $|x_3|=1$.  For $x=(\bar x_2x_3, x_2,x_3) \in b\E$ we have 
\[
 1 = |h_3(x)| = | (1-c)\bar x_2x_3x_2+cx_3| = \left|(1-c)|x_2|^2 + c\right|.
\]
Since this relation holds whenever $|x_2| \leq 1$ we have $c=1$ and hence  $h$ is the identity map.
\end{proof}
\begin{proof}[of Theorem~{\rm\ref{autE}}]
Recall that $x\in\E$ is said to be triangular if $x_1x_2=x_3$.  We denote by $\mathcal{T}$ the set of triangular points.  Note that $x$ is triangular if and only if $\Psi(.,x)$ is a constant map.  It follows that if $x\in\mathcal{T}$ and $\up\in\Aut\D$ then the maps
\[
\Psi(.,\up\cdot x)= \up\circ\Psi(.,x), \quad \Psi(., x\cdot \up)= \Psi(.,x)\circ\up
\]
are constant, and hence $\up\cdot x,\, x\cdot\up \in \mathcal{T}$.  It is clear that $\mathcal{T}$ is invariant under $F$ and under $L_\up, R_\up$  for any $\up\in\Aut\D$, and so $\mathcal{T}$ is invariant under the group $G$ given by equation (\ref{defG}).   Moreover, if $x\in\mathcal{T}$ and we define
\begin{equation} \label{defuchi}
\up(z)= \frac{z-x_1}{\bar x_1 z-1}, \quad \chi(z)=\frac{z-\bar x_2}{x_2 z-1}
\end{equation}
then we find that $\up\cdot x\cdot \chi = O$.  Hence $\mathcal{T}$ is the $G$-orbit of $O$ in $\E$.

Let $V$ denote the orbit of $O$ under $\Aut\E$.  Clearly $V \supset \mathcal{T}$.  By \cite[Satz 1]{kaup}, $V$ is a closed connected complex submanifold of $\E$.  Since $\E$ is inhomogeneous, $V \neq \E$.  If $V$ is a $3$-dimensional submanifold of $\E$ then $V$ is both open and closed in $\E$, and so by connectedness $V=\E$, a contradiction. Thus $V$ is a connected $2$-dimensional submanifold of $\E$.  Since $\mathcal{T}$ is closed in $\E$ it follows that $\mathcal{T}$ is an open and closed subset of $V$, hence by connectedness is equal to $V$.  That is, $\mathcal{T}$ is the orbit of $O$ under $\Aut\E$.  Hence every automorphism of $\E$ restricts to an automorphism of $\mathcal{T}$.

Consider any $f\in\Aut\E$.  Let $f(O)=x$, so that $x\in\mathcal{T}$.  Define $\up,\chi$ as in equations (\ref{defuchi}) and let $g(.)=\up\cdot f(.)\cdot\chi$.  Then $g\in\Aut\E$ and $g(O)=O$.   The restriction $g_\mathcal{T}$ of $g$ to $\mathcal{T}$ is an automorphism of $\mathcal{T}$.  Since $\mathcal{T}$ is isomorphic to the bidisc $\D^2$, $g_\mathcal{T}$ induces an automorphism of $\D^2$ that fixes $(0,0)$.  Hence $g_\mathcal{T}$ is one of the automorphisms
\[
(x_1,x_2,x_1x_2) \mapsto (\omega x_1,\eta x_2,\omega \eta x_1x_2) \mbox{  or  }
(x_1,x_2,x_1x_2) \mapsto (\eta x_2,\omega x_1,\omega \eta x_1x_2)
\]
for some $\omega ,\eta\in\T$.  In the former case let $h(.)=\rho_{\bar\omega }\cdot g(.)\cdot \rho_{\bar\eta}$: then $h\in\Aut\E$ and $h$ fixes $\mathcal{T}$
pointwise.  By Lemma \ref{fixes}, $h$ is the identity map $\mathrm{id}_\E$.  Thus $g=\rho_{\omega }\cdot\mathrm{id}_\E\cdot\rho_{\eta}$ and so
\[
f=L_\up R_\chi g =L_{\up \circ\rho_\omega} R_{\rho_\eta\circ\chi} \in G.
\]
In the latter case a similar argument shows that 
\[
f= L_{\up\circ \rho_\eta} R_{\rho_\omega\circ\chi}F \in G.
\]
Thus in either case $f\in G$.
\end{proof}

\section{The action of $\Aut\E$ on a foliation} \label{leaves}
Condition (v) of Theorem \ref{recap} shows that if $|\beta_1|+|\beta_2| < 1$ then the function
\[
\ph_{\beta_1\beta_2}:\D\to\C^3: \la \mapsto (\beta_1+\bar\beta_2\la, \beta_2+\bar\beta_1\la, \la)
\]
maps $\D$ into $\E$, and moreover every point of $\E$ lies on some disc $\ph_{\beta_1\beta_2}(\D)$.  If $x=\ph_{\beta_1\beta_2}(\la)$ then we find that
\[
\beta_1=\frac{x_1-\bar x_2\la}{1-|\la|^2}, \quad \beta_2=\frac{x_2-\bar x_1\la}{1-|\la|^2},
\]
and so $x$ lies in a {\em unique} disc $\ph_{\beta_1\beta_2}(\D)$.  Thus the discs $\ph_{\beta_1\beta_2}(\D), |\beta_1|+|\beta_2| < 1$, constitute a foliation of $\E$ by analytic discs, which we shall call the {\em $\beta$-foliation} of $\E$.  It is easily checked that $\Psi(\omega,.)$ is an analytic left inverse of $\ph_{\beta_1\beta_2}$ modulo $\Aut\D$ for any $\omega\in\T$, and hence the leaves of the $\beta$-foliation are complex geodesics of $\E$.

The action of $\Aut\E$ can be understood in terms of its action on the $\beta$-foliation.
\begin{theorem} \label{leafaction}
$\Aut\E$ permutes the leaves of the $\beta$-foliation transitively.  Specifically, if $x\in \ph_{\beta_1\beta_2}(\D)$ and $\up, \chi\in\Aut\D$ are given by
\begin{equation} \label{defu&chi}
\up(z)=\omega\frac{z-\al}{\bar\al z - 1}, \quad \chi(z)=\zeta\frac{z-\theta}{\bar\theta z-1}
\end{equation}
then $x\cdot\chi \in \ph_{\gamma_1\gamma_2}(\D)$ where
\begin{equation} \label{defgamma}
\gamma_1=  \frac{\beta_1(1-|\theta|^2)}{|1-\zeta \theta\beta_2|^2-|\theta\beta_1|^2} , \quad \gamma_2 = \frac{\bar\theta(1-|\beta_1|^2+|\beta_2|^2)-\zeta\beta_2-\bar\zeta\bar\theta^2\bar\beta_2}{|1-\zeta\theta\beta_2|^2-|\theta\beta_1|^2}
\end{equation}
and $\up\cdot x\in \ph_{\delta_1\delta_2}(\D)$ where
\begin{equation} \label{defdelta}
\delta_1= \omega\frac{\al(1-|\beta_2|^2+|\beta_1|^2)-\beta_1-\al^2\bar\beta_1}{|1-\bar\al\beta_1|^2-|\al\beta_2|^2},
\quad \delta_2= \frac{\beta_2(1-|\al|^2)}{|1-\bar\al\beta_1|^2-|\al\beta_2|^2} .
\end{equation}
Moreover, for any  $\beta_1,\beta_2$ such that $|\beta_1|+|\beta_2|<1$, if $\up,\chi$ in equations (\ref{defu&chi}) are chosen with $\omega=\zeta=1$,
\begin{eqnarray}\label{chaltheta}
\al &=& \xi_1\tanh\{\tfrac 12 \tanh^{-1}(|\beta_1|+|\beta_2|) + \tfrac 12 \tanh^{-1}(|\beta_1|-|\beta_2|)\},  \\
 \theta &=& \xi_2\tanh\{\tfrac 12 \tanh^{-1}(|\beta_1|+|\beta_2|) - \tfrac 12 \tanh^{-1}(|\beta_1|-|\beta_2|)\} \nn
\end{eqnarray}
where $\beta_1=|\beta_1|\xi_1, \, \bar\beta_2=|\beta_2|\xi_2$ and $\xi_1,\, \xi_2 \in\T$, then
\[
\up^{-1}\cdot\ph_{\beta_1\beta_2}(\D)\cdot\chi^{-1} = \ph_{00}(\D) = \{(0,0,\la): \la\in\D\}.
\]
\end{theorem}
\begin{proof}
A straightforward calculation shows that
\begin{equation} \label{calcgamma}
\ph_{\beta_1\beta_2}(\la)\cdot\chi = \ph_{\gamma_1\gamma_2}(\mu)
\end{equation}
where 
\begin{equation} \label{muetc}
\mu = \eta \frac{\la+c}{\bar c\la+1}, \quad
 \eta = -\zeta\frac{1-\bar\zeta\bar\theta\bar\beta_2}{1-\zeta\theta\beta_2}, \quad
 c = -\frac{\bar\zeta\bar\theta\beta_1}{1-\bar\zeta\bar\theta\bar\beta_2}
\end{equation}
and $\gamma_1,\gamma_2$ are given by equations (\ref{defgamma}).  Similarly
\[
\up\cdot\ph_{\beta_1\beta_2}(\la) = \ph_{\delta_1\delta_2}(\nu)
\]
where 
\begin{equation}\label{nuetc}
\nu = \eta' \frac{\la+c'}{\bar c'\la+1}, \quad
 \eta' = -\omega\frac{1-\al\bar\beta_1}{1-\bar\al\beta_1}, \quad
 c' = -\frac{\al\beta_2}{1-\al\bar\beta_1}
\end{equation}
and $\delta_1,\delta_2$ are given by equations (\ref{defdelta}).  Since $|c|<1,|c'|<1, |\eta|=1$ and $|\eta'|=1$, both $L_\up$ and $R_\chi$ map any $\beta$-leaf bijectively onto another $\beta$-leaf.  As $F$ clearly does likewise, it follows that every automorphism of $\E$ permutes the leaves of the $\beta$-foliation.

 On applying equations (\ref{defdelta}) and (\ref{defgamma}) to the case $\beta_1=\beta_2=0$ we find that
\begin{equation} \label{00}
\up\cdot\ph_{00}(\D)\cdot\chi = \ph_{(\omega\al)0}(\D)\cdot\chi=\ph_{\gamma_1\gamma_2}(\D)
\end{equation}
where 
\begin{equation} \label{nextgam}
\gamma_1=\frac{\omega\al(1-|\theta|^2)}{1-|\al\theta|^2}, \quad
 \gamma_2=\frac{\bar\theta(1-|\al|^2)}{1-|\al\theta|^2}.
\end{equation}
Consider $\beta_1,\beta_2$ such that $|\beta_1|+|\beta_2|<1$.
Choose $\omega=\zeta=1$ and choose $\al,\, \theta$ according to equations (\ref{chaltheta}); note that $|\al|<1,\, |\theta|<1$ since $|\beta_1|+|\beta_2| <1$, and $\xi_1|\al|=\al,\, \bar\xi_2|\theta|=\bar\theta$.  Furthermore
\begin{eqnarray*}
\tanh^{-1}|\al|&=& \tfrac 12 \tanh^{-1}(|\beta_1|+|\beta_2|) + \tfrac 12 \tanh^{-1}(|\beta_1|-|\beta_2|), \\
 \tanh^{-1}|\theta| &=& \tfrac 12 \tanh^{-1}(|\beta_1|+|\beta_2|) - \tfrac 12 \tanh^{-1}(|\beta_1|-|\beta_2|),
\end{eqnarray*}
whence
\[
\tanh^{-1}|\al|+\tanh^{-1}|\theta| =\tanh^{-1}(|\beta_1|+|\beta_2|), \quad
\tanh^{-1}|\al|-\tanh^{-1}|\theta| =\tanh^{-1}(|\beta_1|-|\beta_2|).
\]
On taking $\tanh$ of both sides we obtain
\[
\frac{|\al|+|\theta|}{1+|\al\theta|} = |\beta_1|+|\beta_2|, \quad
\frac{|\al|-|\theta|}{1-|\al\theta|} = |\beta_1|-|\beta_2|,
\]
and therefore
\begin{eqnarray*}
\beta_1&=&\xi_1|\beta_1|=\xi_1\frac{|\al|(1-|\theta|^2)}{1-|\al\theta|^2}=\frac{\al(1-|\theta|^2)}{1-|\al\theta|^2}=\gamma_1\\
\beta_2&=&\bar\xi_2|\beta_2|=\bar\xi_2\frac{|\theta|(1-|\al|^2)}{1-|\al\theta|^2}= \frac{\bar\theta(1-|\al|^2)}{1-|\al\theta|^2}=\gamma_2.
\end{eqnarray*}
Thus $\up\cdot\ph_{00}(\D)\cdot\chi=\ph_{\beta_1\beta_2}(\D)$, as required.  It follows that the action of $\Aut\E$ on the set of $\beta$-leaves is transitive.
\end{proof} 
The theorem shows that the orbit of any point of $\E$ under $\Aut\E$ contains a point of the form $(0,0,\la)$ with $\la\in\D$; the application of a further rotation shows that we may take $0\leq\la<1$.  The calculations above allow us to be precise.
\begin{theorem}
Let $x\in\E$.  The orbit of $x$ under $\Aut\E$ contains a unique point of the form $(0,0,r)$ with $r\in [0,1)$.  If $x=(\beta_1+\bar\beta_2\la,\beta_2+\bar\beta_1\la,\la)$ then $r$ is given by
\[
 r = \left|\frac{\la-\al\bar\theta}{\bar\al\theta\la-1}\right|
\]
where $\al,\, \theta$ are given by equations (\ref{chaltheta}).
\end{theorem}
\begin{proof}
 As in equations (\ref{00}) and (\ref{nextgam}) we have, for $\up, \chi$ given by equations (\ref{defu&chi}) and $z\in\D$,
\[
\up\cdot\ph_{00}(z)\cdot\chi = \ph_{(\omega\al)0}(\nu)\cdot\chi=\ph_{\gamma_1\gamma_2}(\mu)
\]
where (by equations (\ref{muetc}), (\ref{nuetc}))
\[
 \nu= -\omega z, \quad \mu = \omega\zeta\frac{ z+\bar\zeta\al\bar\theta}{\zeta\bar\al\theta z+1}.
\]
Now choose $\omega=\zeta=1$ and choose $\al, \theta$ as in equations (\ref{chaltheta}).
As we showed above, $\gamma_1=\beta_1$ and $\gamma_2=\beta_2$.  Thus
\[
\up\cdot(0,0,z)\cdot\chi=\ph_{\beta_1\beta_2}\left(\frac{ z+\al\bar\theta}{\bar\al\theta z+1}\right).
\]
 Substitute $z=(\la-\bar\al\theta)/(-\al\bar\theta\la+1)$ and apply a suitable rotation $\rho$ to obtain
\[
\up\cdot\rho\cdot(0,0,r)\cdot\chi = \ph_{\beta_1\beta_2}(\la)
\]
with $r$ as in the theorem.

It remains to prove the uniqueness of $r$.   Suppose that $(0,0,r), (0,0,s)$ both lie in the orbit of $x$ with $0\leq r,s<1$; then there exist $\up, \chi \in\Aut\D$ such that $\up\cdot(0,0,r) =(0,0,s)\cdot\chi$ (we can ignore $F$ here since both points are fixed by $F$).
That is, if $\up, \chi$ are given by equations (\ref{defu&chi}), $\ph_{(\omega\al)0}(-\omega r)=\ph_{0\bar\theta}(-\zeta s)$.   It follows that $\al=\theta=0$ and $\omega r=\zeta s$.  Since $r,s \geq 0$ we have $r=s$.
\end{proof}
In \cite[Theorem 3.4.4]{alaa} it is shown by a different method that, for every $x\in\E$, there exist $\up,\chi\in\Aut\D$ such that $\up\cdot x\cdot\chi= (0,0,r)$ for some $r\in[0,1)$;  different formulae for $\up, \chi$ and $r$ are obtained.

\section{Concluding remarks} \label{conclud}
The original purpose for the study of both the tetrablock and the symmetrised bidisc  $\G$ was to try and solve special cases of the $\mu$-synthesis problem (\cite[Section 9]{awy}, \cite{alaa,AY1,AY2}) which is a refinement of classically-studied interpolation problems.  Although the approach has indeed led to some new results which are relevant to the motivating engineering problem, the results so far are too special to be of great import in applications.   It is reasonable to hope that a better understanding of the complex geometry of these domains and analogous ones will in the future provide results that will be very useful for the theory of $H^\infty$ control.  Meanwhile, the study of $\G$ has proved to be of considerable interest to specialists in several complex variables (e.g. \cite{EZ,JP}; numerous authors have developed the theory of $\G$ and its higher-dimensional analogues further).  The appeal of these domains is that they admit a rich and explicit function theory that is in some ways close to that of classical domains, such as Cartan domains, but in others has new and subtle features.  The present paper is more ``several complex variables'' than ``$H^\infty$ control'': it addresses some of the basic questions one would ask about any domain of interest.   It does however have some implications for cases of the $\mu$-synthesis problem.
For example, suppose we are given $2\times 2$ complex matrices $A, B$ with $A=[A_1 \,  A_2]$ strictly triangular (but not $0$) and $B=[B_1 \, B_2]=[b_{ij}]$ not diagonal, and we are asked whether there exists an analytic matrix function $F$ in the unit disc such that $F(0)=A, F'(0)=B$ and $\mu(F(\la))\leq 1$ for all $\la\in\D$ (here $\mu$ is a certain cost function lying between the spectral radius and the operator norm; see \cite[Section 9]{awy}).  It follows from Theorem \ref{newschw} above that such an $F$ exists if and only if $\max\{|b_{11}|,|b_{22}|\} + |A_1\wedge B_2+A_2\wedge B_1| \leq 1$.  Such explicit criteria for $\mu$-synthesis problems are hard to come by in general.  Again, knowledge of the automorphisms of $\E$ reveals a non-obvious equivalence between certain $\mu$-synthesis problems.

 \noindent N. J. Young\\
School of Mathematics\\
Leeds University,
England


\begin{thebibliography}{99}
%
%

\bibitem{alaa}
{ A. A. Abouhajar}, {\em  Function theory related to $H^\infty$ control,}
Ph.D. thesis, Newcastle University, 2007.

\bibitem{awy}
{ A. A. Abouhajar, M. C. White and N. J. Young},  A Schwarz lemma for a domain related to mu-synthesis, arXiv:0708.0637.

\bibitem{AY1} 
{ J. Agler and N. J. Young}, The two-by-two spectral Nevanlinna-Pick problem, {\em Trans. Amer. Math. Soc.}, Vol. 356, No. 2 (2004), 573-585.


\bibitem{AY2} 
{ J. Agler and N. J. Young},  The hyperbolic geometry of the symmetrized bidisc, {\em  J. Geom. Anal.}, Vol. 14, No. 3 (2004), 375-403.

\bibitem{BFT} 
{ H. Bercovici, C. Foia\c{s} and A. Tannenbaum}, A spectral commutant lifting theorem, {\em Trans. Amer. Math. Soc.}, Vol. 325, No. 2 (1991), 741-763.

\bibitem{EZ}
A. Edigarian and W. Zwonek, Geometry of the symmetrised polydisc, {\em Archiv  Math.}, 84(2005), 364-374.

\bibitem{fuks}
{ B. A. Fuks}, {\em Special chapters in the theory of analytic functions of several complex variables}, Translations of Mathematical Monographs {\bf 14} (American Mathematical Society, Providence RI, 1965).

 \bibitem{harris}
 { L. A. Harris},  Bounded symmetric homogeneous domains in infinite-dimensional spaces,  in {\em Proceedings on Infinite Dimensional Holomorphy}, Springer Lecture Notes in Mathematics {364} (1973), 13-40. 

\bibitem{harris2}
 { L. A. Harris}, Analytic invariants and the Schwarz-Pick inequality,
 {\em Israel J. Math.} 34 (1979), 177-197.

\bibitem{JP}
{ M. Jarnicki and P. Pflug},  On automorphisms of the symmetrised bidisc,
 {\em Arch. Math. (Basel)}  83  (2004), no. 3, 264--266.

 \bibitem{kaup}
{ W. Kaup}, \"Uber das Randverhalten von holomorphen Automorphismen beschr\"ankte Gebiete, {\em Manuscripta Math.} 3 (1970), 257-270.

 \bibitem{krantz}
{ S. G. Krantz}, {\em Function theory of several complex variables},  Wiley, New York, 1982.

\bibitem{par}
{ S. Parrott},  On a quotient norm and the Sz.-Nagy-Foia\c{s} lifting theorem, {\em J. Functional Analysis} 30 (1978), 311-328.

\bibitem{Y}
{ N. J. Young}, {\em An introduction to Hilbert space},  Cambridge University Press, Cambridge, 1988.
\end{thebibliography}
\end{document}